%% file: HomCLQCat.tex
\documentclass{amsart}
\usepackage{amsthm,amsfonts,amssymb,amsmath,amscd,graphics
	,newclude, enumitem,geometry,stmaryrd, fancyhdr
}
\usepackage{textcomp}
\usepackage{appendix}
\usepackage{xcolor}
\usepackage{longtable}
\usepackage{url}
\usepackage[centerlast]{subfigure}
\usepackage[all,2cell, cmtip]{xy} \UseAllTwocells

\begin{document}
	\begin{titlepage}
		\begin{center}
			\textbf{\uppercase{Compact closed categories and $\Gamma$-categories }}\\[0.015cm]
			\textit{(with an appendix by André Joyal)} \\[1cm]
			\Large{\uppercase{Amit Sharma}}\\[0.25cm]
		
		\end{center}

		ABSTRACT: In this paper we study compact closed categories within the context of homotopical algebra.
		We construct two new model category structures by localizing two (Quillen equivalent) model categories of symmetric monoidal categories with the objective of establishing the free compact closed category on one generator as a fibrant replacement of the free symmetric monoidal category on one generator, in our localized model categories. We go on to show that the fibrant objects in our model categories are compact closed categories.


	\end{titlepage}

\date{\today}
\input{my_definitions}
		\tableofcontents

	\include{Introduction}
	\include{CCPerm}
	\include{CCGCat}

	\appendix

\include{Duality}
	\include{LocMdlCats}

	\bibliographystyle{amsalpha}
	\bibliography{HomCLQCat}
	
\end{document}

%% file: my_definitions.tex
%

\newcommand{\CONT}{\noindent}
\newcommand{\FIG}{Fig.\ }
\newcommand{\FIGS}{Figs.\ }
\newcommand{\SEC}{Sec.\ }
\newcommand{\SECS}{Secs.\ }
\newcommand{\TAB}{Table }
\newcommand{\TABS}{Tables }
\newcommand{\EQ}{Eq.\ }
\newcommand{\EQS}{Eqs.\ }
\newcommand{\APP}{Appendix }
\newcommand{\APPS}{Appendices }
\newcommand{\CHP}{Chapter }
\newcommand{\CHPS}{Chapters }

\newcommand{\OFF}{\emph{G2off}~}
\newcommand{\TOO}{\emph{G2Too}~}
\newcommand{\CatS}{Cat_{\bigS}}
\newcommand{\PicS}{{\underline{\pic}}^{\oplus}}
\newcommand{\HPicS}{{Hom^{\oplus}_{\pic}}}

\newtheorem{thm}{Theorem}[section]
\newtheorem{lem}[thm]{Lemma}
\newtheorem{conj}[thm]{Conjecture}
\newtheorem{coro}[thm]{Corollary}
\newtheorem{prop}[thm]{Proposition}

\theoremstyle{definition}
\newtheorem{df}[thm]{Definition}
\newtheorem{nota}[thm]{Notation}

\newtheorem{ex}[thm]{Example}
\newtheorem{exs}[thm]{Examples}

\theoremstyle{remark}
\newtheorem*{note}{Note}
\newtheorem{rem}{Remark}
\newtheorem{ack}{Acknowledgments}

\newcommand{\ChI}{{\textit{\v C}}\textit{ech}}
\newcommand{\Ch}{{\v C}ech}

\newcommand{\ChZG}{hermitian line $0$-gerbe}
\renewcommand{\theack}{$\! \! \!$}

\newcommand{\ChG}{flat hermitian line $1$-gerbe}
\newcommand{\ChC}{hermitian line $1$-cocycle}
\newcommand{\ChGG}{flat hermitian line $2$-gerbe}
\newcommand{\ChCC}{hermitian line $2$-cocycle}
\newcommand{\id}{id}
\newcommand{\LC}{\mathfrak{C}}
\newcommand{\Coker}{Coker}
\newcommand{\Com}{Com}
\newcommand{\Hom}{Hom}
\newcommand{\Mor}{Mor}
\newcommand{\Map}{Map}
\newcommand{\alg}{alg}
\newcommand{\an}{an}
\newcommand{\Ker}{Ker}
\newcommand{\Ob}{Ob}
\newcommand{\Proj}{\mathbf{Proj}}
\newcommand{\topo}{\mathbf{Top}}
\newcommand{\kan}{\mathcal{K}}
\newcommand{\pkan}{\mathcal{K}_\bullet}
\newcommand{\Kan}{\mathbf{Kan}}
\newcommand{\pKan}{\mathbf{Kan}_\bullet}
\newcommand{\QCat}{\mathbf{QCat}}
\newcommand{\gp}{\mathcal{A}_\infty}
\newcommand{\mdl}{\mathcal{M}\textit{odel}}
\newcommand{\sSets}{\mathbf{sSets}}
\newcommand{\sSetsQ}{(\mathbf{sSets, Q})}
\newcommand{\sSetsK}{(\mathbf{sSets, \Kan})}
\newcommand{\pSSets}{\mathbf{sSets}_\bullet}
\newcommand{\pSSetsK}{(\mathbf{sSets}_\bullet, \Kan)}
\newcommand{\pSSetsQ}{(\mathbf{sSets_\bullet, Q})}
\newcommand{\cyl}{\mathbf{Cyl}}
\newcommand{\lin}{\mathcal{L}_\infty}
\newcommand{\Vect}{\mathbf{Vect}}
\newcommand{\Aut}{Aut}
\newcommand{\pic}{\mathcal{P}\textit{ic}}
\newcommand{\Dlin}{\pic}
\newcommand{\bigS}{\mathbf{S}}
\newcommand{\bigA}{\mathbf{A}}
\newcommand{\bhom}{\mathbf{hom}}
\newcommand{\bhomK}{\mathbf{hom}({\textit{K}}^+,\textit{-})}
\newcommand{\Bhom}{\mathbf{Hom}}
\newcommand{\bhomk}{\mathbf{hom}^{{\textit{k}}^+}}
\newcommand{\Dlino}{\pic^{\textit{op}}}
\newcommand{\lino}{\mathcal{L}^{\textit{op}}_\infty}
\newcommand{\lind}{\mathcal{L}^\delta_\infty}
\newcommand{\linK}{\mathcal{L}_\infty(\kan)}
\newcommand{\linC}{\mathcal{L}_\infty\text{-category}}
\newcommand{\linCs}{\mathcal{L}_\infty\text{-categories}}
\newcommand{\ainCs}{\text{additive} \ \infty-\text{categories}}
\newcommand{\ainC}{\text{additive} \ \infty-\text{category}}
\newcommand{\inC}{\infty\text{-category}}
\newcommand{\inCs}{\infty\text{-categories}}
\newcommand{\gS}{{\Gamma}\text{-space}}
\newcommand{\gSet}{{\Gamma}\text{-set}}
\newcommand{\ggS}{\Gamma \times \Gamma\text{-space}}
\newcommand{\gSs}{\Gamma\text{-spaces}}
\newcommand{\gSets}{\Gamma\text{-sets}}
\newcommand{\ggSs}{\Gamma \times \Gamma\text{-spaces}}
\newcommand{\gO}{\Gamma-\text{object}}
\newcommand{\gSCat}{{\Gamma}\text{-space category}}
\newcommand{\pss}{\mathbf{S}_\bullet}
\newcommand{\gSC}{{{{\Gamma}}\mathcal{S}}}
\newcommand{\pGSC}{{{{\Gamma}}\mathcal{S}}_\bullet}
\newcommand{\pGSCStr}{{{{\Gamma}}\mathcal{S}}_\bullet^{\textit{str}}}
\newcommand{\ggSC}{{\Gamma\Gamma\mathcal{S}}}
\newcommand{\gSD}{\mathbf{D}(\gSC^{\textit{f}})}
\newcommand{\sCat}{\mathbf{sCat}}
\newcommand{\pSCat}{\mathbf{sCat}_\bullet}
\newcommand{\gSetCat}{{{{\Gamma}}\mathcal{S}\textit{et}}}
\newcommand{\Dhom}{\mathbf{R}Hom_{\pic}}
\newcommand{\gop}{\Gamma^{\textit{op}}}
\newcommand{\fU}{\mathbf{U}}
\newcommand{\cDN}{\underset{\mathbf{D}[\textit{n}^+]}{\circ}}
\newcommand{\cDK}{\underset{\mathbf{D}[\textit{k}^+]}{\circ}}
\newcommand{\cDL}{\underset{\mathbf{D}[\textit{l}^+]}{\circ}}
\newcommand{\cD}{\underset{\gSD}{\circ}}
\newcommand{\cDT}{\underset{\gSD}{\widetilde{\circ}}}
\newcommand{\ppsSets}{\sSets_{\bullet, \bullet}}
\newcommand{\gdHom}{\underline{Hom}_{\gSD}}
\newcommand{\HomU}{\underline{Hom}}
\newcommand{\ominf}{\Omega_\infty}
\newcommand{\ev}{ev}
\newcommand{\cu}{C(X;\mathfrak{U}_I)}
\newcommand{\Sing}{Sing}
\newcommand{\AlgEin}{\A\textit{lg}_{\E_\infty}}
\newcommand{\SFunc}[2]{\mathbf{SFunc}({#1} ; {#2})}
\newcommand{\unit}[1]{\mathrm{1}_{#1}}
\newcommand{\liminj}{\varinjlim}
\newcommand{\limproj}{\varprojlim}
\newcommand{\HMapC}[3]{\mathcal{M}\textit{ap}^{\textit{h}}_{#3}(#1, #2)}
\newcommand{\tensPGSR}[2]{#1 \underset{\gSR}\wedge #2}
\newcommand{\pTensP}[3]{#1 \underset{#3}\wedge #2}
\newcommand{\MGCat}[2]{\underline{\map}_{\gCAT}({#1},{ #2})}
\newcommand{\MGCatGen}[3]{\underline{\map}_{#3}({#1},{ #2})}
\newcommand{\MGBoxCat}[2]{\underline{\map}_{\gSC}^{\Box}({#1},{ #2})}
\newcommand{\TensPFunc}[1]{- \underset{#1} \otimes -}
\newcommand{\TensP}[3]{#1 \underset{#3}\otimes #2}
\newcommand{\MapC}[3]{\mathcal{M}\textit{ap}_{#3}(#1, #2)}
\newcommand{\bHom}[3]{{#2}^{#1}}
\newcommand{\gn}[1]{\Gamma^{#1}}
\newcommand{\gnk}[2]{\Gamma^{#1}({#2}^+)}
\newcommand{\gnf}[2]{\Gamma^{#1}({#2})}
\newcommand{\ggn}[1]{\Gamma\Gamma^{#1}}
\newcommand{\Nat}{\mathbb{N}}
\newcommand{\partition}[2]{\delta^{#1}_{#2}}
\newcommand{\inclusion}[2]{\iota^{#1}_{#2}}
\newcommand{\EinQC}{\text{coherently commutative monoidal quasi-category}} 
\newcommand{\EinQCs}{\text{coherently commutative monoidal quasi-categories}}
\newcommand{\pHomCat}[2]{[#1,#2]_{\bullet}}
\newcommand{\CatHom}[3]{[#1,#2]^{#3}}
\newcommand{\pCatHom}[3]{[#1,#2]_\bullet^{#3}}
\newcommand{\EinC}{\text{coherently commutative monoidal category}}
\newcommand{\EinCs}{\text{coherently commutative monoidal categories}}
\newcommand{\EinLO}{E_\infty{\text{- local object}}}
\newcommand{\EinSLO}{\E_\infty\S{\text{- local object}}}
\newcommand{\Ein}{E_\infty}
\newcommand{\EinS}{E_\infty{\text{- space}}}
\newcommand{\EinSs}{E_\infty{\text{- spaces}}}
\newcommand{\PCat}{\mathbf{Perm}}
\newcommand{\PCatCC}{\mathbf{Perm}^{\textit{cc}}}
\newcommand{\nor}[1]{{#1}^\textit{nor}}
\newcommand{\pSSetsHom}[3]{[#1,#2]_\bullet^{#3}}
\newcommand{\PNat}{\overline{\L}}
\newcommand{\PStr}{\L}
\newcommand{\Gn}[1]{\Gamma[#1]}
\newcommand{\GIH}{\Gamma\textit{H}_{\textit{in}}}
\newcommand{\QStr}[1]{\L_\bullet(\ud{#1})}
\newcommand{\QStrF}{\L_\bullet}
\newcommand{\Kbar}{\overline{\K}}
\newcommand{\gPerm}{{\Gamma\PCat}}
\newcommand{\gCat}{{\Gamma{\text{-category}}}}
\newcommand{\Du}{{\mathfrak{D}}}
\newcommand{\EHom}[2]{[#1; #2]^\E}
\newcommand{\Cob}[1]{{\mathfrak{Fr}}^{\textit{cc}}(\ud{#1})}
\newcommand{\gpd}{{\mathbf{Gpd}}}
\newcommand{\ol}[1]{\overline{#1}}
\newcommand{\gCob}[1]{\Gamma\Cob{#1}}
\newcommand{\gCatCC}{{\gCAT^{\textit{cc}}}}
\newcommand{\gCatSM}{{\gCAT^{\otimes}}}
\newcommand{\gPCat}{\Gamma\PCat}
\newcommand{\pGPCat}{\gPCat_\bullet}
\newcommand{\lCC}{\PStr^{\textit{cc}}}
\newcommand{\kCC}{\K^{\textit{cc}}}
\newcommand{\Comp}{\mathbf{Comp}^{\textit{str}}}
\newcommand{\Supp}[1]{{\textit{Supp}(#1)}}

\def\F{\mathcal F}
\def\Pic{\mathbf{2}\mathcal P\textit{ic}}
\def\nc{\mathbb C}

\def\Z{\mathbb Z}
\def\P{\mathbb P}
\def\J{\mathcal J}
\def\I{\mathcal I}
\def\nC{\mathbb C}
\def\H{\mathcal H}
\def\A{\mathcal A}
\def\C{\mathcal C}
\def\D{\mathcal D}
\def\E{\mathcal E}
\def\G{\mathcal G}
\def\B{\mathcal B}
\def\L{\mathcal L}
\def\U{\mathcal U}
\def\K{\mathcal K}

\def\M{\mathcal M}
\def\O{\mathcal O}
\def\R{\mathcal R}
\def\S{\mathcal S}
\def\N{\mathcal N}

\newcommand{\Fin}{\F\textit{in}}
\newcommand{\undertilde}[1]{\underset{\sim}{#1}}
\newcommand{\abs}[1]{{\lvert#1\rvert}}
\newcommand{\mC}[1]{\mathfrak{C}(#1)}
\newcommand{\sigInf}[1]{\Sigma^{\infty}{#1}}
\newcommand{\x}[4]{\underset{#1, #2}{ \overset{#3, #4} \prod }}
\newcommand{\mA}[2]{\textit{Add}^n_{#1, #2}}
\newcommand{\mAK}[2]{\textit{Add}^k_{#1, #2}}
\newcommand{\mAL}[2]{\textit{Add}^l_{#1, #2}}
\newcommand{\Mdl}[2]{\L_\infty}
\newcommand{\inv}[1]{#1^{-1}}
\newcommand{\Lan}[2]{\mathbf{Lan}_{#1}(#2)}
\newcommand{\ccCat}{\mathbf{cc}\PCat}
\newcommand{\fCC}{\F^{\textit{cc}}}

\newcommand{\del}{\partial}
\newcommand{\sCatO}{\mathcal{S}Cat_\O}
\newcommand{\FCgop}{\mathbf{F}\mC{N(\gop)}}
\newcommand{\hProd}{{\overset{h} \oplus}}
\newcommand{\hProdn}{\underset{n}{\overset{h} \oplus}}
\newcommand{\hProdk}[1]{\underset{#1}{\overset{h} \oplus}}
\newcommand{\map}{\mathcal{M}\textit{ap}}
\newcommand{\SMGS}[2]{\map_{\gSC}({#1},{ #2})}
\newcommand{\MGS}[2]{\underline{\map}_{\gSC}({#1},{ #2})}
\newcommand{\MGSBox}[2]{\underline{\map}^{\Box}_{\gSC}({#1},{ #2})}
\newcommand{\Aqcat}[1]{\underline{#1}^\oplus}
\newcommand{\Cat}{\mathbf{Cat}}
\newcommand{\Sp}{\mathbf{Sp}}
\newcommand{\SpStb}{\mathbf{Sp}^{\textit{stable}}}
\newcommand{\SpStr}{\mathbf{Sp}^{\textit{strict}}}
\newcommand{\Sspec}{\mathbb{S}}
\newcommand{\ud}[1]{\underline{#1}}
\newcommand{\inrt}{\mathbf{Inrt}}
\newcommand{\act}{\mathbf{Act}}
\newcommand{\StrSMHom}[2]{[#1,#2]_\otimes^{\textit{str}}}
\newcommand{\SMHom}[2]{[#1,#2]_\otimes}
\newcommand{\ESMHom}[2]{[#1,#2]_\otimes^\E}
\newcommand{\gCats}{\Gamma\text{-categories}}
\newcommand{\gCAT}{\Gamma\Cat}
\newcommand{\gCATCCM}{\Gamma\Cat^\otimes}
\newcommand{\KStr}[1]{\K(#1)}
\newcommand{\pGPCatStr}{\pGPCat^{\textit{Str}}}
\newcommand{\KSeg}{\K}

%% file: Introduction.tex
\section{Introduction}
\label{Introduction}

A \emph{compact closed} category is a symmetric monoidal category having the special property that each object has a left (and therefore a right) dual. The archetype example of compact closed categories is the category of finite dimensional vector spaces. Some other prominent examples of compact closed categories include the category of finitely generated projective modules over a commutative ring and the category of finite dimensional representations of a compact group. The category of abelian groups can be characterized as a \emph{reflective} localization of the category of commutative monoids, namely the localization functor has a right adjoint which is the fully faithful inclusion of the full subcategory of abelian groups (local-objects). It happens that this localization is generated by a single map which is the inclusion $i_+:\Nat \to \Z$.
The model category of (permutative) Picard groupoids $(\PCat, \pic)$ \cite{sharma5} is a left Bousfield localization of the natural model category of permutative (or strict symmetric monoidal) categories $\PCat$. This localization is also generated by a single map which is the inclusion $i^{\pic}:\F^\otimes(\ast) \to \F^{\pic}{(\ast)}$ of the free permutative category on one generator into the free Picard groupoid on one generator. In this paper we obtain a model category of (permutative) compact closed categories as a left Bousfield localization of the category of permutative categories $\PCat$.
The localization is again generated by a single map which is the inclusion $i^{}:\F^\otimes(\ast) \to \Cob{1}$,  where the codomain is the free compact closed category on one generator. The main objective of this paper is to compare the aforementioned model category of (permutative) compact closed categories with a model category of \emph{coherently} compact closed categories which is a left Bousfield localization of a (model) category of coherently commutative monoidal categories \cite{Sharma} generated by a single map which is (up to equivalence) an adjunct of the generator $i^{}$.
The main result of this paper is that the two aforementioned model categories are Quillen equivalent. This result maybe regarded as a \emph{strictification theorem} for (coherently) compact closed categories.

 The classical ($1$-dimensional) cobordism hypothesis \cite{BD95},\cite{JL1} informally states that the (framed) $1$-Bordism category, namely the category whose objects are framed $0$-dimensional manifolds and morphisms are (diffeomorphism classes of) framed $1$-dimensional manifolds with boundary, is the free compact closed category on one generator. To a purely algebraic problem, the cobordism hypothesis provides an answer which is rooted in differential topology. In this paper we seek an answer to the same underlying algebraic problem within homotopical algebra. This paper is a first in a series of papers aimed at developing a theory for compact closed $(\infty, n)$-categories and also providing a purely algebraic description of a free compact closed $(\infty, n)$-category on one generator. In this paper we construct a model category whose fibrant objects can be described as categories equipped with a coherently commutative multiplication structure wherein each object has a dual. This model category is intended to be a prototype for subsequently constructing model categories whose fibrant objects are models for $(n+k, n)$-categories equipped with a coherently commutative monoidal structure and which are \emph{fully dualizable}.

 
  Normalized coherently commutative monoidal categories were introduced in the paper \cite{segal} where they were called $\gCats$. These (normalized) objects have also been referred to in the literature as \emph{special} $\gCats$. A model category whose fibrant objects are (unnormalized) coherently commutative monoidal categories was constructed in \cite{Sharma}. In this paper we will denote this model category by $\gCatSM$. Unlike a symmetric monoidal category, higher coherence data is specified as a part of the definition of a coherently commutative monoidal category. Moreover, in the latter, a \emph{tensor product} of two objects is unique only up to a contractible space of choices. In this paper we extend the notion of compact closed categories to the more generalized setting of coherently commutative monoidal categories and define \emph{coherently compact closed categories}.
  We construct another model category structure on the (functor) category $\gCAT$ whose fibrant objects are coherently compact closed categories. This model category, denoted $\gCatCC$, is a (left) Bousfield localization of the model category of coherently commutative monoidal categories $\gCatSM$.
We go on to show that the thickened Segal's Nerve functor \cite[Sec. 6.]{Sharma}  is a right Quillen functor of a Quillen equivalence between the aforementioned model category structure of (permutative) compact closed categories on $\PCat$ and $\gCatCC$. The following picture depicts the idea of coherently compact closed categories and also depicts how various coherently commutative objects in $\Cat$ are related:
\begin{equation*}
\xymatrix@R=16mm{
	& *+++++[e][F]{\txt{Categories}} \ar[d]^{\txt{free construction}}  \\
	& *+++++[e][F]{\txt{Coherently \\ commutative monoidal \\ categories}} \ar[ld]_{\txt{adding \\ inverses} \ \ } \ar[rd]^{\txt{adding \\
			duals}} \\
	*+++++[e][F]{\txt{Coherently \\ commutative Picard \\ groupoids}} &&
	*+++++[e][F]{\txt{Coherently  \\ compact closed \\ categories}}
}
\end{equation*}
The addition processes depicted by the two diagonal arrows in the above picture are manifested by localizations.

The Barrat-Priddy-Quillen theorem was reformulated in the language of $\gSs$ in \cite{segal}. In the same paper, Segal constructed a functor from the category of (normalized) $\gSs$ $\pGSC$ to the category of (connective) spectra. This functor maps the unit of the symmetric monoidal structure on $\pGSC$, namely the free $\gS$ on one generator $\gn{1}$, to the \emph{sphere spectrum}. In section $2$ of the same paper, Segal constructed a (normalized) $\gS$, which he denoted by $B\Sigma$, which can also be described as (simplicial nerve of) the (categorical) Segal's nerve \cite{Sharma} of the (skeletal) permutative category of finite sets  and bijections, denoted $\KStr{\ud{\N}}$. The reformulated theorem states that the spectrum associated to the $\gS$ $B\Sigma$ is \emph{stably} equivalent to the sphere spectrum. In other words, the reformulation states that the $\gS$ $\gn{1}$ is equivalent to $B\Sigma$ in the \emph{stable} model category of $\gSs$ constructed in \cite{Schwede}. A stronger version of this theorem called the (special) Barrat-Priddy-Quillen theorem appeared in the paper \cite{BMoer}. This theorem states that the two $\gSs$ in context are also \emph{unstable} equivalent \emph{i.e.} they are equivalent in a model category of special $\gSs$. Along the same lines, our construction of the model category $\gCatCC$ implies that the Segal's nerve of the free compact closed category on  $1$-generator $\KSeg{(\Cob{1})}$ is a fibrant replacement of the (representable) $\gCat$ $\gn{1}$ in the model category $\gCatCC$.

In Appendix A we collect some folklore results regarding the notion of duality. In Appendix B we recall the basic notion and existence result of left Bousfield localization of model categories.

\begin{ack}
	The author would like to thank André Joyal for having numerous discussions with the author regarding this paper and also for writing Appendix A: Aspects of Duality, which has added a lot of clarity to the paper. 
	
	\end{ack}

%% file: CCPerm.tex
\section[Compact closed permutative categories]{Compact closed permutative categories}
\label{cc-mdl-str-Perm}
A compact closed category is a symmetric monoidal category wherein each object has the special property of having a left (and hence a right) \emph{dual}. The category of all (small) symmetric monoidal categories has a subcategory $\PCat$ which inherits a model category structure from the \emph{natural} model category $\Cat$. The objects of $\PCat$ are \emph{permutative} categories (or strict symmetric monoidal categories) which are those symmetric monoidal categories whose \emph{tensor product} is strictly associative and strictly unital. The morphisms of $\PCat$ are strict symmetric monoidal functors namely those symmetric monoidal functors which preserve the symmetric monoidal structure strictly.
In this section we will construct another model category structure on the category of permutative categories $\PCat$ whose fibrant objects are (permutative) compact closed categories. We begin by recalling the definition of a compact closed category from \cite{KL}:
\begin{df}
	\label{CC-cats}
	A \emph{compact closed} category is a symmetric monoidal category $C$ in which each object $c \in Ob(C)$ can be assigned a triple $(c^\bullet, \eta_c, \epsilon_c)$ where $c^\bullet$ is an object of $C$ (called \emph{right dual} of $c$) and $\eta_c:\unit{C} \to c^\bullet \otimes c$ and $\epsilon_c:c \otimes c^\bullet \to \unit{C}$ are two maps in $C$ such that the following two maps are identities:
	\begin{equation}
	\label{left-dual}
	c^\bullet  \cong \unit{C} \otimes c^\bullet \overset{\eta_c \otimes id} \to c^\bullet \otimes c \otimes c^\bullet \overset{id \otimes \epsilon_c } \to c^\bullet \otimes \unit{C} \cong c^\bullet 
	\end{equation}
	and
	\begin{equation}
	\label{right-dual}
	c  \cong c \otimes \unit{C} \overset{id \otimes \eta_c} \to c \otimes c^\bullet \otimes c \overset{\epsilon_c \otimes  id} \to c \otimes \unit{C} \cong c 
	\end{equation}
\end{df}

\begin{rem}
	The symmetric monoidal structure ensures that the right dual is also a left dual and therefore we will just call $c^\bullet$ as the dual of $c$.
\end{rem}
\begin{nota}
	Unless specified otherwise, in this paper a compact closed category will mean a permutative category which is compact closed.
\end{nota}

We recall that a compact closed category $C$ is a closed symmetric monoidal category wherein the internal hom object, between two objects $c_1, c_2 \in C$ is defined as follows:
\[
[c_1, c_2]_C := c_1 \otimes c_2^\bullet.
\]

 In the paper \cite{KL} the authors describe a category $\Comp$ whose objects are pairs $(C, L(C))$ consisting if  a compact closed permutative category $C$ and a collection:
 \[
 L(C) :=  \underset{A \in Ob(C)}\bigsqcup (A^\bullet, \eta_A, \epsilon_A) 
 \]
 where $A^\bullet$ is an object of $C$ and $\eta_A$ and $\epsilon_A$ are the unit and counit maps which exhibit $A^\bullet$ as a dual of $A$ in $C$. In other words, an object of $\Comp$ is a permutative category $C$ along with a chosen \emph{duality data} $(A^\bullet, \eta_A, \epsilon_A)$ for each object $A$. The morphisms of $\Comp$ are those strict symmetric monoidal functors which (strictly) preserve duality data. The category $\Comp$ is equipped with an obvious forgetful functor $U^{\textit{cc}}:\Comp \to \PCat$. We denote by $U'$ the following composite of forgetful functors:
 \[
 \Comp \overset{U^{\textit{cc}}} \to \PCat \overset{U} \to \Cat.
 \]
 In the same paper the following adjunction is constructed:
 \[
 \F':\Cat \rightleftarrows \Comp:U'
 \]
 The authors go on to construct a \emph{stricter} version $\F''$ of the left adjoint functor $\F'$ mentioned above which is equipped with a natural equivalence $\beta:\F'' \Rightarrow \F$ where a component of $\beta$ is an inclusion functor.
 We want to provide an easier description of the category $\F''(\ud{1})$, where $\ud{1}$ is the terminal category. We will construct a compact closed permutative category $\Cob{1}$ which is isomorphic to $\F''(\ud{1})$.
 \begin{nota}
 	\label{set-of-signs}
 	We denote by $\ol{1}$ the following set:
 	 \[
 	 \ol{1} := \lbrace +, - \rbrace
 	 \]
 	 and refer to it as the set of orientations of a point.
 	\end{nota}
 
   The object set of $\Cob{1}$ is defined as follows:
\[
Ob(\Cob{1}) :=  \underset{n \in \Nat} \bigsqcup  \Fin(\ud{n}; \ol{1}) ,
\]
where $\Fin$ is the category of (unbased) finite sets and (unbased) maps and $\ud{n} = \lbrace 1, 2 \dots, n \rbrace$.

The set $\ol{1}$ is equipped with a bijection $in:\ol{1} \to \ol{1}$ which changes the sign \emph{i.e.}  $in(+) = -$ and $in(-) = +$. For each $\ud{n}$ is equipped with a bijection $\Sigma_\textit{rev}(n):\ud{n} \to \ud{n}$ which reverses the order: $i \mapsto n-i + 1$.
\begin{df}
	\label{def-dual}
	For each object $f:\ud{n} \to \ol{1}$ in the proposed permutative category $\Cob{1}$, we define another object which we denote by $f^\bullet$, by the following composite:
	\[
	\ud{n} \overset{\Sigma_\textit{rev}(n)} \to \ud{n} \overset{f} \to \ol{1} \overset{in} \to \ol{1}.
	\]
	We will refer to $f^\bullet$ as the \emph{dual} of $f$.
\end{df}
\begin{nota}
	For an object $f:\ud{n} \to \ol{1}$ in the proposed permutative category $\Cob{1}$, the \emph{length} of $f$ is the cardinality of its domain namely $n$.
	\end{nota}
\begin{df}
	A \emph{tensor product} of $f:\ud{n} \to \ol{1}$ with $g:\ud{m} \to \ol{1}$ is the canonical map inducted on the sum of domains $\ud{n+m}$ which is denoted:
	\[
	f \Box g:\ud{n + m} \to \ol{1}
	\]
	\end{df}
The above tensor product has a unit object which we define next:
\begin{df}
	We denote by $\ud{0}:\emptyset \to \ol{1}$ the unique map from the emptyset to $\ol{1}$. We will refer to $\ud{0}$ as the \emph{unit} object of $\Cob{1}$.
	\end{df}
  \begin{df}
  	A \emph{pairing} $G$ from $f:\ud{n} \to \ol{1}$ to $g:\ud{m} \to \ol{1}$ in the proposed category $\Cob{1}$ is a finite collection $G = \lbrace s_i: i \in \ud{p} \rbrace$ of sections of the (canonical) map:
  	\[
  	f^\bullet \Box g:\ud{n+m} \to \ol{1}
  	\]
  	such that the following conditions are satisfied:
  	\begin{enumerate}
  		\item The images of the sections in the collection $G$ are pairwise disjoint subsets of $\ud{n+m}$.
  		\item The collection $G$ is maximal in the sense that for any section $t$ of $f^\bullet \Box g$ such that $t \notin G$, there exists an $s_j \in G$ such that the intersection $im(t) \cap im(s_j)$ is non-empty.
  		\end{enumerate}
  	\end{df}
  
  \begin{df}
  	In a pairing $G = \lbrace s_i: i \in \ud{p} \rbrace$ from $f:\ud{n} \to \ol{1}$ to $g:\ud{m} \to \ol{1}$, a section $s_i$ is called a \emph{external} section if $im(s_i) \nsubseteq \ud{n}$ and $im(s_i) \nsubseteq \ud{m}$. Otherwise the section is called an \emph{internal} section.
  	\end{df}
  \begin{df}
  	An internal section $s_i$ in a pairing $G = \lbrace s_i: i \in \ud{p} \rbrace$ from $f:\ud{n} \to \ol{1}$ to $g:\ud{m} \to \ol{1}$ is called a \emph{codomain} section if it factor's through the domain of the codomain object $g$ as follows:
  	\[
  	\xymatrix{
  	 \ol{1} \ar[rr]^{s_i} \ar[rd] && \ud{n+m} \\
  	 & \ud{m} \ar@{^{(}->}[ru]
     }
  	\]
  	A pairing is called a pairing of \emph{codomain index} $k$  if it contains $k$ codomain sections.
  	\end{df}
  \begin{df}
  	An internal section $s_i$ in a pairing $G = \lbrace s_i: i \in \ud{p} \rbrace$ from $f:\ud{n} \to \ol{1}$ to $g:\ud{m} \to \ol{1}$ is called a \emph{domain} section if it factor's through the domain of the domain object $f$ as follows:
  	\[
  	\xymatrix{
  		\ol{1} \ar[rr]^{s_i} \ar[rd] && \ud{n+m} \\
  		& \ud{n} \ar@{^{(}->}[ru]
  	}
  	\]
  	A pairing is called a pairing of \emph{domain index} $k$  if it contains $k$ domain sections.
  \end{df}
\begin{df}
	A \emph{total} pairing  from $f:\ud{n} \to \ol{1}$ to $g:\ud{m} \to \ol{1}$ in the proposed category $\Cob{1}$ is a pairing $G = \lbrace s_i: i \in \ud{p} \rbrace$ such that there is a bijection:
	\[
	\underset{i \in \ud{p}} \bigsqcup im(s_i) \cong \ud{n + m}
	\] 
\end{df}
\begin{nota}
	We will denote a total pairing $G$ from $f:\ud{n} \to \ol{1}$ to $g:\ud{m} \to \ol{1}$ by $G:f \leadsto g$ and refer to $f$ as the \emph{domain} of $G$ and $g$ as the \emph{codomain} of $G$.
	\end{nota}
\begin{df}
	\label{pairing-nums}
	We associate with a total pairing $G$, a triple $(G_e, G_d, G_c)$ consisting of natural numbers, where $G_e$ is the number of external sections in $G$, $G_d$ is the number of domain sections in $G$ and $G_c$ is the number of codomain sections in $G$.
	\end{df}
  Now we describe a \emph{composition} of two total pairings $G$ from $f:\ud{n} \to \ol{1}$ to $g:\ud{m} \to \ol{1}$ and $H$ from $g:\ud{m} \to \ol{1}$ to $h:\ud{l} \to \ol{1}$. We observe that $G$ and $H$ determine two bijections:
  	\[
\underset{i \in \ud{(n+m)/2}} \bigsqcup \ol{1} \cong  \underset{i \in \ud{(n+m)/2}} \bigsqcup im(s_i) \cong \ud{n + m}
  \] 
  and
  	\[
  \underset{i \in \ud{(m+p)/2}} \bigsqcup \ol{1} \cong \underset{i \in \ud{(m+p)/2}} \bigsqcup im(t_i) \cong \ud{m + p}
  \] 
  These two bijections together determine another bijection
  \[
  [H, G]:\underset{i \in \ud{(n+2m+p)/2}} \bigsqcup \ol{1} \cong \ud{n + 2m + p}
  \]
  Now we consider the following pullback diagram:
  \[
  \xymatrix{
  \underset{i \in \ud{(n+p)/2}} \bigsqcup \ol{1} \ar[r] \ar[d]_{(H, G)}^\cong & \underset{i \in \ud{(n+2m+p)/2}} \bigsqcup \ol{1} \ar[d]^{[H, G]} \\
  \ud{n + p} \ar[r] & \ud{n + 2m + p}
   }
  \]
  The bijection $(H, G)$ uniquely determines a total pairing which we denote by $H \circ G$.
  \begin{df}
  	\label{cir-created}
  	For two composable total pairings $G$ and $H$ as above, we define a natural number $\text{Cir}(H, G)$ as follows:
  	\[
  	\textit{Cir}(H, G) := H_c + G_c - (H \circ G)_c
  	\]
  	We will refer to this number as the \emph{circles created} by the composition of $H$ and $G$.
  	\end{df}
 
  Now we have all the machinery needed to define a morphism in $\Cob{1}$:
  \begin{df}
  	A morphism from $f:\ud{n} \to \ol{1}$ to $g:\ud{m} \to \ol{1}$ in $\Cob{1}$ is a pair $(G, k)$ where $G$ is a total pairing from $f$ to $g$ and $k \in \Nat$.
  	\end{df}
  Let $(H, l):g \to h$ be another morphism in $\Cob{1}$. We define
  \[
  (H, l) \circ (G, k) := (H \circ G, k+l + \textit{Cir}(H, G))
  \]
  \begin{nota}
  	Let $s:\ol{1} \to \ud{n}$ be a section of an object $f:\ud{n} \to \ol{1}$. We denote by $(s, m)$ the following map obtained by extending the codomain of $s$:
  	\[
  	\ol{1} \overset{s} \to \ud{n} \hookrightarrow \ud{n+m}.
  	\]
  	We denote by $(m, s)$ the following map:
  	\[
  	\ol{1} \overset{s} \to \ud{n} \hookrightarrow \ud{m+n}.
  	\]
  	\end{nota}
  \begin{rem}
  	Let $s:\ol{1} \to \ud{n}$ be a section of a map $f:\ud{n} \to \ol{1}$ which represents an object in $\Cob{1}$ and $g:\ud{m} \to \ol{1}$ be another object in $\Cob{1}$, then $(s, m):\ol{1} \to \ud{n+m}$ is a section of $f \Box g:\ud{n+m} \to \ol{1}$.
  	\end{rem}
  The \emph{tensor product} of two maps $(G, k):f \to g$ and $(J, q):g \to h$ is defined as follows:
  \[
  (G, k) \Box (J, q) := (G \Box H, k + q),
  \]
  where the total pairing $G \Box H$ is defined as follows:
  \[
  \lbrace (s_i, m):  s_i \in G \rbrace \bigsqcup \lbrace (n, t_i):  t_i \in J \rbrace.
  \]
 The aforementioned tensor product defines a bifunctor:
 \[
 - \Box - :\Cob{1} \times \Cob{1} \to \Cob{1}.
 \]
 The above bifunctor endows the category $\Cob{1}$ with a strict symmetric monoidal (permutative) category structure.
 Further, the permutative category $\Cob{1}$ is a compact closed category wherein each object $f:\ud{n} \to \ol{1}$ has a (two-sided) dual $f^\bullet$ namely, there exist maps $\eta_f:\ud{0} \to f \Box f^\bullet$ and $\epsilon_f:f^\bullet \Box f \to \ud{0}$ in $\Cob{1}$ such that the equations \eqref{left-dual} and \eqref{right-dual} are satisfied.
 
 \begin{rem}
 	\label{Strict-CC}
 	In the compact closed category $\Cob{1}$, we observe the following:
 	\begin{enumerate}
 		\item The unit object is its own dual \emph{i.e.} $\ud{0}^\bullet = \ud{0}$.
 		\item The dual of the dual of an object $f:\ud{n} \to \ol{1}$ in $\Cob{1}$, is the object $f$ itself \emph{i.e.} $(f^\bullet)^\bullet = f$, for each $f \in \Cob{1}$.
 		
 		\item For a pair of objects $f:\ud{n} \to \ol{1}$ and $g:\ud{m} \to \ol{1}$ in $\Cob{1}$, the following commutative diagram implies the equality $(f \Box g)^\bullet = g^\bullet \Box f^\bullet$:
 		\begin{equation*}
 		\xymatrix@C=25mm{
 		 \ud{m+n} \ar[r]^{\Sigma_\textit{rev}(\ud{m+n})} \ar[rd]_{\Sigma_\textit{rev}(\ud{m}) + \Sigma_\textit{rev}(\ud{n}) \ \ \ \ \ \ } & \ud{n + m} \ar[r]^{f \Box g} & \ol{1} \ar[d]_{\textit{in}} \\
 		 & \ud{m+n} \ar[r] \ar[ru]_{g \Box f} & \ol{1}
         }
 		\end{equation*}
 		\end{enumerate}
 	\end{rem}
 
 \begin{rem}
 	\label{univ-prop-cob}
 	The compact closed permutative category $\Cob{1}$ is equipped with an isomorphism of permutative categories $can:\Cob{1} \cong U^{\textit{cc}}(\F''(\ud{1}))$. Further there is a strict symmetric monoidal functor:
 	\[
 	\Cob{1} \overset{can} \to U^{\textit{cc}}(\F''(\ud{1})) \overset{U^{\textit{cc}}(\beta(\ud{1}))} \to U^{\textit{cc}}(\F'(\ud{1}))
 	\]
 	which is an equivalence of permutative categories. Now the univeral property of $\F'(\ud{1})$ implies that for each object $c$ of a compact closed permutative category $C$ there is a strict symmetric monoidal functor $F_c:\Cob{1} \to C$ which is unique upto a unique natural isomorphism.
 	\end{rem}
 
\begin{rem}
	\label{equiv-with-fr-bord}
	The permutative category $\Cob{1}$ is equivalent to the symmetric monoidal category whose objects are $0$-dimensional framed manifolds and morphisms are (boundary preserving) diffeomorphism classes of $1$-dimensional framed manifolds with boundary. In other words $\Cob{1}$ is an algebraic model for the framed $1$-Bordism category.
\end{rem}
\begin{df}
	\label{Alg-Bord-cat}
	We will refer to $\Cob{1}$ either as the \emph{free} compact closed category on one generator or as the algebraic $1$-Bordism category. 
\end{df}

\begin{rem}
	\label{Fr-SM-CC-incl}
	The free compact closed category over one generator $\Cob{1}$ is equipped with an inclusion (strict) symmetric monoidal functor $i:\F^\otimes(\ud{1}) \to \Cob{1}$, where $\F^\otimes(\ud{1})$ is the free permutative category on one generator namely it is (isomorphic to) the category of finite sets and bijections. Both $\Cob{1}$ and $\F^\otimes(\ud{1})$ are cofibrant objects in the natural model category of permutative categories $\PCat$.
\end{rem}
\begin{prop}
	\label{CC-cats-loc-obj}
	A permutative category is compact closed if and only if it is a $\lbrace i \rbrace$-local object.
\end{prop}
\begin{proof}
	For any compact closed category $C$, it follows from remark \ref{univ-prop-cob} and the argument in the proof of proposition \ref{1-D-CH} that the following map, which is the evaluation map on the generator, is an equivalence of groupoids:
	\[
	J\left(\StrSMHom{i}{C}\right):J\left(\StrSMHom{\Cob{1}}{C}\right) \to J\left(\StrSMHom{\F^\otimes(\ud{1})}{C}\right) \cong J(C)
	\]
	where $J$ is the right adjoint to the inclusion map $\iota:\gpd \to \Cat$.
	Thus each permutative compact closed category is an $\lbrace i \rbrace$-local object. 
	Conversely, let us assume that $C$ is an $\lbrace i \rbrace$-local object. We recall that for any category $D$ we have the following equality of object sets: $Ob(D) = Ob(J(D))$. By assumption the functor $J\left(\StrSMHom{i}{C}\right)$ is an equivalence of groupoids which now implies that each object of $C$ is isomorphic to some object in the image of $J\left(\StrSMHom{i}{C}\right)$. Now the result follows from the observation that each object in the image of $J\left(\StrSMHom{i}{C}\right)$ has a dual.
\end{proof}

\begin{df}
	A map of permutative categories $F:C \to D$ will be called a compact closed equivalence of permutative categories if it is a $\lbrace i \rbrace$-local equivalence.
\end{df}

\begin{rem}
	A strict symmetric monoidal functor $F:C \to D$ between cofibrant permutative categories is a compact closed equivalence if
	the following functor is an equivalence of groupoids:
	\[
	J(\StrSMHom{F}{E}):J(\StrSMHom{D}{E}) \to J(\StrSMHom{C}{E})
	\]
	for each permutative compact closed category $E$.
\end{rem}

Now we state and prove the main result of this section which is is regarding the construction of a model category of compact closed categories. The proof uses an existence theorem of a left Bousfield localization of a (class of) model category which is reviewed in appendix \ref{loc-Mdl-Cats}:
\begin{thm} 
	\label{nat-model-str-Perm}
	There is a model category structure on the category of all small
	permutative categories and strict symmetric monoidal functors $\PCat$ in which
	\begin{enumerate}
		\item A cofibration is a strict symmetric monoidal functor which is a cofibration in the natural model category $\PCat$.
		\item A weak-equivalence is a compactly closed equivalence of permutative categories.
		\item A fibration is a strict symmetric monoidal functor having the right lifting property with respect to all maps which are both cofibrations and weak equivalences.
	\end{enumerate}
	Further, this model category structure is combinatorial and left-proper.
\end{thm}
\begin{proof}
	The model category structure is a left Bousfield localization of the left-proper, combinatorial natural model category structure on $\PCat$ with respect to a single map $i:\F^\otimes(\ud{1}) \to \Cob{1}$.
	The existence of this left Bousfield localization and the characterization of cofibrations and weak-equivalences follows from \ref{local-tool}.
	\end{proof}
	\begin{nota}
		We denote the above model category by $\PCatCC$.
		\end{nota}

%% file: CCGCat.tex
\section[Coherently compact closed categories]{Coherently compact closed categories}
\label{cc-gCat}
In this section we will construct another model category  structure on the category of $\gCats$ $\gCAT = [\Gamma^{op}, \Cat]$. The main result of this section is that the thickened Segal's nerve functor $\Kbar$ is a right Quillen functor of a Quillen equivalence between the model category of coherently compact closed categories, which will be constructed in this section, and the model category of compact closed permutative categories $\PCatCC$ constructed in the previous section. We construct the desired model category as a left Bousfield localization of the model category of coherently commutative monoidal categories constructed in \cite{Sharma}, which we denote by $\gCATCCM$.
%
We begin by briefly recalling that the thickened Segal's nerve functor $\Kbar:\PCat \to \gCAT$ constructed in \cite{Sharma}:
\begin{df}
	\label{P-Nat-n}
	  \begin{sloppypar}
	For each $n \in \Nat$ we will now define a permutative groupoid $\PNat(n)$.
		We will denote an object of this groupoid by $(f_1, f_2, \dots, f_r)$. A morphism $(f_1, f_2, \dots, f_r) \to (g_1, g_2, \dots, g_k)$ is an
		isomorphism of finite sets
		\[
		F:\Supp{f_1} \sqcup \Supp{f_2} \sqcup \dots \sqcup \Supp{f_r} \overset{\cong} \to \Supp{g_1} \sqcup \Supp{g_2} \sqcup \dots \sqcup \Supp{g_k}
		\]
		such that the following diagram commutes
	\end{sloppypar}
	\begin{equation*}
	\xymatrix{
		\Supp{f_1} \sqcup \dots \sqcup \Supp{f_r} \ar[rr]^{F} \ar[rd] && \Supp{g_1} \sqcup \dots \sqcup \Supp{g_k} \ar[ld] \\
		&\ud{n}
	}
	\end{equation*}
	where the diagonal maps are uniquely determined by the inclusions of components of the coproducts into $\ud{n}$ and $\Supp{f}$ denotes the support of the based map $f$.
\end{df}
The (thickened) Segal's nerve functor is now defined, in degree $n$ as follows:
\[
\Kbar(C)(n^+) := \StrSMHom{\PNat(n)}{C},
\]
where $C$ is a permutative category. The functor $\Kbar:\PCat \to \gCAT$ has a left adjoint, denoted $\PNat$, see \cite[Sec. 6]{Sharma}. 
The thickened Segal's nerve of the free compact closed category on one generator $\Cob{1}$, denoted by $\Kbar({\Cob{1}})$, is equipped with an inclusion map 
\begin{equation}
\label{inc-j}
j:\gn{1} \to \Kbar({\Cob{1}})
\end{equation}

 which is determined by the generator of $\Cob{1}$. 
 
 \begin{rem}
 	\label{equiv-with--Seg-Ner-fr-bord}
 	The coherently commutative monoidal category $\Kbar({\Cob{1}})$ is equivalent to the thickened Segal's nerve of the framed $1$-Bordism (symmetric monoidal) category described in remark \eqref{equiv-with-fr-bord}.
 \end{rem}

%
%
\begin{df}
	\label{cc-gCategory}
	A coherently commutative monoidal category $X$ is called a \emph{ coherently compact closed} category if the symmetric monoidal category $X(1^+)$ is a (not necessarily permutative) compact closed category.
	\end{df}

\begin{df}
	We will refer to the coherently compact closed category $\Kbar({\Cob{1}})$ as the thickened Segal's nerve of the algebraic $1$-Bordism category.
\end{df}

\begin{df}
	A $\lbrace j \rbrace$-local equivalence will be called a compact closed equivalence of $\gCats$.
	\end{df}
\begin{thm}
	\label{mdl-CC-GCat}
	There is a left-proper, combinatorial model category structure on the category $\gCAT$ wherein a map is a
	\begin{enumerate}
		\item cofibration if it is a strict cofibration of $\gCats$, namely a cofibration in the strict model category of $\gCats$.
		
		\item weak-equivalence if it is a compact closed equivalence of $\gCats$.
		\item a fibration if it has the right lifting property with respect to maps which are simultaneously cofibrations and weak-equivalences.
		\end{enumerate}
	
	\end{thm}
\begin{proof}
	The model category structure is obtained by carrying out a left Bousfield localization of the natural model category structure on $\PCat$ with respect to $\lbrace j \rbrace$, this follows from \cite[Thm. 2.11]{CB1}. The same theorem implies that the model category is combinatorial and left-proper.
	\end{proof}
\begin{nota}
	We denote the above model category by $\gCatCC$ and refer to it as the model category of coherently compact closed categories.
	\end{nota}

\begin{lem}
	\label{Quil-Pair-KSeg-PNat}
	The adjoint pair $(\PNat, \Kbar)$ is a Quillen pair between the model category $\PCatCC$ and the model category $\gCatCC$.
\end{lem}
\begin{proof}
	We recall from above that the model category $\gCatCC$ is a left Bousfield localization of the  model category of coherently commutative monoidal categories $\gCatSM$ therefore it has the same cofibrations as $\gCatSM$, namely $Q$-cofibrations. Since the adjoint pair in context is a Quillen pair between $\PCat$ and $\gCatSM$ therefore the left adjoint $\PNat$ preserves cofibrations between the two model categories in the context of the theorem. The fibrations between fibrant objects in $\gCatCC$ are strict fibrations of $\gCats$ which are preserved by $\Kbar$. Now \cite[Prop. E.2.14]{AJ1} tells us that the adjoint pair $(\PNat, \Kbar)$ is a Quillen pair between $\PCatCC$ and $\gCatCC$.
\end{proof}
\begin{rem}
	\label{K-L-Str-QE}
	An argument similar to the proof of the above lemma shows that the adjoint pair $(\PStr, \KSeg)$ defined in \cite[Sec. 5]{Sharma} is a also a Quillen pair.
\end{rem}
\begin{thm}
	\label{char-CC-gCat}
	A coherently commutative monoidal category $X$ is a $\lbrace j \rbrace$-local object if and only if $\PNat(X)$ is a compact closed permutative category.
\end{thm}
\begin{proof}
	Let us first assume that $\PNat(X)$ is a compact closed category, then $\Kbar\PNat(X)$ is a $\lbrace j \rbrace$-local object because the adjunction $(\PNat, \Kbar)$ is a Quillen pair, by lemma \eqref{Quil-Pair-KSeg-PNat}, and a right Quillen functor preserves fibrant objects. Further, the unit map $\eta_X:X \to \Kbar\PNat(X)$ is a strict equivalence of $\gCats$ \cite[lem. 6.15]{Sharma}. This implies that $X$ is a $\lbrace j \rbrace$-local object because $\Kbar\PNat(X)$ is one.
	
	Conversely, let us assume that $X$ is a $\lbrace j \rbrace$-local object. We consider the following commutative diagram:
	\begin{equation*}
	\xymatrix@C=24mm{
		\HMapC{\Kbar(\Cob{1})}{\Kbar\PNat(X)}{\gCatCC} \ar[r]^{ \ \ \ \ \HMapC{j}{\Kbar\PNat(X)}{\gCatCC}} \ar[d]_\cong & \HMapC{\gn{1}}{\Kbar\PNat(X)}{\gCatCC} \ar[d]^\cong \\
		\HMapC{\PNat(\Kbar(\Cob{1}))}{\PNat(X)}{\PCatCC} \ar[r]_{ K } & \HMapC{\PNat(\gn{1})}{\PNat(X)}{\PCatCC}  
	}
	\end{equation*}
	where the bottom horizontal map $K = \HMapC{\PNat(j)}{\PNat(X)}{\PCatCC}$. Since $X$ is a $\lbrace j \rbrace$-local object by assumption, the $\gCat$ $\Kbar\PNat(X)$ is also one. Thus the top row is a homotopy equivalence of Kan complexes. By the two out of three property of weak-equivalences in a model category, $K$ is a homotopy equivalence of Kan-complexes.
 Now the following commutative diagram implies that $\PNat(X)$ is a compact closed category:
	\begin{equation*}
	\xymatrix@C=32mm{
		\HMapC{\PNat(\Kbar\Cob{1})}{\PNat(X)}{\PCatCC} 
		\ar[r]^{ \ \ \ \ \HMapC{\PNat(j)}{\PNat(X)}{\PCatCC}}  & \HMapC{\PNat(\gn{1})}{\PNat(X)}{\PCatCC} \ar[d]^T \\
		\HMapC{\Cob{1}}{\PNat(X)}{\PCatCC}  \ar[u]^{\HMapC{\epsilon}{\PNat(X)}{\PCatCC}} \ar[r] & \HMapC{\F(\ud{1})}{\PNat(X)}{\PCatCC}
	}
	\end{equation*}
	where $\epsilon$ is the counit map which is a weak equivalence in the natural model category $\PCat$. The downward map $T = \HMapC{\iota}{\PNat(X)}{\PCatCC}$ is a homotopy equivalence of Kan complexes because the inclusion functor $\iota:\F(\ud{1}) \to \PNat(\gn{1})$ is an equivalence of categories.
	
\end{proof}

\begin{coro}
	\label{j-loc-ccc-cat}
	An $\gCat$ $X$ is a fibrant object of the model category $\gCatCC$ if and only if it is a coherently compact closed $\gCat$.
\end{coro}
\begin{proof}
	We begin by recalling that for each coherently commutative monoidal category $Z$, its degree one category $Z(1^+)$ inherits a symmetric monoidal category structure \cite[Prop. 3.3.1.]{Leinster}.
	A $\gCat$ $X$ is fibrant in the model category $\gCatCC$ if and only if it is a coherently commutative monoidal category and a $\lbrace j \rbrace$-local object. In this case, the unit map $\eta_X:X \to \Kbar(\PNat(X))$ is a strict equivalence of $\gCats$. Now we have we have the following commutative diagram of equivalence of categories:
	\[
	\xymatrix{
	X(1^+) \ar[r]^{\eta_X(1^+) \ \ \ \ \ \ } \ar[rd]_{i_X} & \Kbar(\PNat(X))(1^+)  \ar[d]^{ev}\\
	& \PNat(X) 
    }
	\]
 This diagram implies that $X(1^+)$ is compact closed because it is equivalent to a compact closed category $\PNat(X)$ via the (symmetric monoidal) functor $i_X$.	

\end{proof}

\begin{coro}
	\label{char-coh-cc-cat}
	The thickened Segal's nerve $\Kbar({C})$ of a compact closed permutative category $C$ is a coherently compact closed category.
\end{coro}
\begin{proof}
	The counit map of the adjunction $(\PNat, \Kbar)$ is a natural equivalence of permutative categories therefore $\epsilon_C:\PNat(\Kbar(C)) \to C$ is a weak equivalence in the natural model category $\PCat$. This implies that $\PNat(\Kbar(C))$ is compact closed. Now the corollary follows from the theorem and the observation that $\Kbar(C)$ is a coherently commutative monoidal category.
	
	\end{proof}

\begin{rem}
	The model category of coherently commutative monoidal categories $\gCatSM$ is a $\Cat$-model category. This implies that
	\[ \HMapC{j}{X}{\gCatSM} = J(\MapC{j}{X}{\gCAT}).
	\]
	Now it is easy to see that following statements are equivalent:
	\begin{enumerate}
		\item The inclusion map $j:\gn{1} \to \Kbar(\Cob{1})$ is a weak-equivalence in $\gCatCC$.
		
		\item For any coherently compact closed category $X$, the following map is an equivalence of function spaces:
		\begin{multline*}
		J\MapC{j}{X}{\gCAT}:J\MapC{\Kbar(Cob{1})}{X}{\gCAT} \to J(\MapC{\gn{1}}{X}{\gCAT}) \\ \overset{\cong} \to JN(X(1^+)).
		\end{multline*}
	\end{enumerate}

	\begin{sloppypar}
		\end{sloppypar}
	
\end{rem}

\begin{prop}
	\label{CC-GCat-class-upto-eq}
	A morphism of $\gCats$ $F:X \to Y$ is a compact closed equivalence of $\gCats$ if and only if for each compact closed (permutative) category $Z$ we have the following homotopy equivalence of function complexes:
	\[
	\HMapC{F}{\Kbar(Z)}{\gCatSM}:\HMapC{Y}{\Kbar(Z)}{\gCatSM} \to \HMapC{X}{\Kbar(Z)}{\gCatSM}
	\]
\end{prop}
\begin{proof}
	$F$ is a compact closed equivalence of $\gCats$ if and only if for each coherently compact closed category $W$, the following map is a homotopy equivalence of Kan complexes:
	\[
	\HMapC{F}{W}{\gCatSM}:\HMapC{Y}{W}{\gCatSM} \to \HMapC{X}{W}{\gCatSM}
	\]
	Since $\Kbar(Z)$ is a coherently compact closed category, by the above corollary, one direction of the statement is obvious.
	
	The other direction of the statement follows from the observation that for each coherently compact closed category $W$, the unit map $\eta_W:W \to \Kbar\PNat(W)$ is a strict equivalence of $\gCats$.
\end{proof}

 It was shown in \cite[Thm. 6.18]{Sharma} that the adjoint pair $(\PNat, \Kbar)$ is a Quillen equivalence between the natural model category of permutative categories $\PCat$ and the model category of coherently commutative monoidal categories $\gCatCC$. This Quillen equivalence is a strict equivalence of the underlying homotopy theories, namely both functors preserve and reflect weak-equivalences of the model categories in context and the unit and the counit maps are natural weak-equivalences.

\begin{lem}
	\label{Seg-Ner-pres-ref-CC-Eq}
	The Segal's nerve functor $\KSeg$ preserves and reflects compact closed equivalences of permutative categories.
	\end{lem}
\begin{proof}
	Let $F:C \to D$ be a compact closed equivalence of permutative categories. It follows from \cite[Thm. 6.18]{Sharma}, \cite[Thm. 6.17]{Sharma} and the observation that each object of the natural model category $\PCat$ is fibrant, that for each permutative category $C$, the counit map $\epsilon_C:\PNat\Kbar(C) \to C$ is a weak-equivalence in the natural model category $\PCat$ \emph{i.e.} the underlying functor of $\epsilon_C$ is an equivalence of categories. We consider the following diagram of function complexes for each compact closed permutative category $Z$:
	\begin{equation*}
	\xymatrix@C=28mm{
	\HMapC{D}{Z}{\PCatCC} \ar[r]^{\HMapC{F}{Z}{\PCatCC}} \ar[d]_{\HMapC{\epsilon_D}{Z}{\PCatCC}} & \HMapC{C}{Z}{\PCatCC} \ar[d]^{\HMapC{\epsilon_C}{Z}{\PCatCC}} \\
	\HMapC{\PNat\Kbar(D)}{Z}{\PCatCC} \ar[r] \ar[d]_\cong & \HMapC{\PNat\Kbar(C)}{Z}{\PCatCC} \ar[d]^\cong \\
	\HMapC{\Kbar(D)}{\Kbar(Z)}{\gCatCC} \ar[r]_{\HMapC{\Kbar(F)}{\Kbar(Z)}{\gCatCC}}  & \HMapC{\Kbar(C)}{\Kbar(Z)}{\gCatCC} 
   }
	\end{equation*}
	The two vertical isomorphisms follow from  \cite[Prop. 17.4.16]{Hirschhorn} applied to the Quillen pair from lemma \ref{Quil-Pair-KSeg-PNat}. It follows from \cite[Thm. 17.7.7]{Hirschhorn} that the top horizontal arrow and the upper two verical arrows are homotopy equivalences of simplicial sets. Now the two out of three property of weak equivalences in a model category implies that the lower horizontal map, namely $\HMapC{\Kbar(F)}{\Kbar(Z)}{\gCatCC}$ is a homotopy equivalence of simplicial sets. Now proposition \ref{CC-GCat-class-upto-eq} and \cite[Thm. 17.7.7]{Hirschhorn} together imply that $\Kbar(F)$ is a weak-equivalence in $\gCatCC$.
	
	Conversely, let us assume that $\Kbar(F)$ is a weak-equivalence in $\gCatCC$. Now the bottom horizontal arrow in the above diagram is a homotopy equivalence of simplicial sets and therefore the top horizontal arrow is one too. Now proposition \ref{CC-GCat-class-upto-eq} and \cite[Thm. 17.7.7]{Hirschhorn} together imply that $F$ is a weak-equivalence in $\PCatCC$. 
	\end{proof}
 The following corollary is an easy consequence of the above lemma:
 \begin{coro}
 	\label{PNat-pres-ref-WE}
 	The left Quillen functor $\PNat$ preserves and reflects compact closed equivalences of $\gCats$.
 	\end{coro}

Now we will state and prove the main result of this paper:
\begin{thm}
	\label{Quil-Eq}
	The adjoint pair $(\PNat, \Kbar)$ is a Quillen equivalence between the natural model category of compact closed permutative categories $\PCatCC$ and the model category of coherently commutative monoidal categories $\gCatCC$.
\end{thm}
\begin{proof}
	Let $X$ be a cofibrant object in $\gCatSM$ and $C$ be a fibrant object of $\PCatCC$. We will show that a map $F:\PNat(X) \to C$ in $\PCatCC$ is a weak equivalence if and only if it's adjunct map $\phi(F):X \to \Kbar(C)$ is a weak-equivalence in $\gCatCC$.
	
	We first recall that the Quillen pair $(\PNat,\Kbar)$ is a Quillen equivalence between the natural model category $\PCat$ and the model category of coherently commutative monoidal categories $\gCatSM$ \cite[Thm. 6.18]{Sharma}. We further recall that every object in the natural model category $\PCat$ is fibrant. Now it follows from \cite[Prop. 1.3.13.]{Hovey} that for each Q-cofibrant $\gCat$ $X$, the unit map of the adjunction $\eta_X:X \to \Kbar\PNat(X)$ is a weak-equivalence in $\gCatSM$ and therefore a weak-equivalence in $\gCatCC$.
	
	Now the result follows from proposition \ref{Seg-Ner-pres-ref-CC-Eq} and the following commutative diagram:
	\begin{equation*}
	\xymatrix{
	\Kbar\PNat(X) \ar[r]^{\Kbar(F)}  & \Kbar(C) \\
	X \ar[ru]_{\phi(F)} \ar[u]^{\eta(X)}
     }
	\end{equation*}
		\end{proof}
	 The following corollary follows from the above theorem, remark \ref{K-L-Str-QE} and the natural weak-equivalence $\KSeg \Rightarrow \Kbar$ constructed in \cite[Cor. 6.19]{Sharma}:
	 
	 \begin{coro}
	 	\label{Quil-Eq-PStr-KSeg}
	 	The adjoint pair $(\PStr, \KSeg)$ is a Quillen equivalence between the natural model category of compact closed permutative categories $\PCatCC$ and the model category of coherently commutative monoidal categories $\gCatCC$.
	 \end{coro}
  The main result is stronger than what is stated in Theorem \ref{Quil-Eq}:
  \begin{rem}
  	The Quillen pair $(\PNat, \Kbar)$ induces a strict equivalence of the underlying homotopy theories on the two model categories in context. More precisely, both functors preserve weak-equivalences and the unit and counit maps are natural weak-equivalences.
  	\end{rem}

%% file: Duality.tex
\makeatletter
\def\@tocline#1#2#3#4#5#6#7{\relax
\ifnum #1>\c@tocdepth 
  \else 
    \par \addpenalty\@secpenalty\addvspace{#2}%
\begingroup \hyphenpenalty\@M
    \@ifempty{#4}{%
      \@tempdima\csname r@tocindent\number#1\endcsname\relax
 }{%
   \@tempdima#4\relax
 }%
 \parindent\z@ \leftskip#3\relax \advance\leftskip\@tempdima\relax
 \rightskip\@pnumwidth plus4em \parfillskip-\@pnumwidth
 #5\leavevmode\hskip-\@tempdima #6\nobreak\relax
 \ifnum#1<0\hfill\else\dotfill\fi\hbox to\@pnumwidth{\@tocpagenum{#7}}\par
 \nobreak
 \endgroup
  \fi}
\makeatother

 \setcounter{tocdepth}{1}


\SelectTips{cm}{}
\newdir{ >}{{}*!/-7pt/@{>}}
\newcommand{\xycenter}[1]{\vcenter{\hbox{\xymatrix{#1}}}}
 \makeatletter

\renewcommand{\part}{\@startsection
  {part}
  {0}
  {0mm}
  {2\baselineskip}
  {1\baselineskip}
  {\centering \Large\sc}}

 \renewcommand{\thepart}{\Roman{part}} 

\makeatother


\newcommand{\ie}{\text{i.e.\ }}
\newcommand{\myemph}{\textit} 



\newtheoremstyle{mythm}%
{5pt}
{}
{\itshape}
{}
{\bfseries}
{.}
{.5em}
{}%

\newtheoremstyle{mydef}%
{5pt}
{3pt}
{}
{}
{\bfseries}
{.}
{.5em}
{}%

\theoremstyle{mythm}
\newtheorem{theorem}{Theorem}[section] 
\newtheorem{lemma}[theorem]{Lemma} 
\newtheorem{proposition}[theorem]{Proposition} 
\newtheorem{corollary}[theorem]{Corollary}  
\newtheorem{apptheorem}{Theorem}[section]
\newtheorem{exercise}[theorem]{Exercise}

\theoremstyle{mydef}
\newtheorem{definition}[theorem]{Definition}	
\newtheorem*{definition*}{Definition}	
\newtheorem{remark}[theorem]{Remark} 
\newtheorem*{remark*}{Remark} 
\newtheorem{example}[theorem]{Example}
\newtheorem{examples}[theorem]{Examples}
\newtheorem*{example*}{Example}
\newtheorem*{examples*}{Examples}

\newtheorem{para}[theorem]{}


\newcommand\pfun{\mathrel{\ooalign{\hfil$\mapstochar\mkern5mu$\hfil\cr$\to$\cr}}}

\newcommand{\defeq}{=_{\mathrm{def}}}
\newcommand{\co}{\colon}
\newcommand{\iso}{\cong} 
\newcommand{\rev}{\mathit{\vee}}
\newcommand{\op}{\mathrm{op}}
\newcommand{\catequiv}{\simeq} 
\newcommand{\cateq}{\simeq} 
\newcommand{\coend}{\int}  
\newcommand{\mat}{\multimap}
\newcommand{\asym}{\rightsquigarrow}
\newcommand{\HOM}{\mathrm{HOM}}
\newcommand{\HOMp}{\mathrm{HOM}^p}
\newcommand{\End}{\mathbf{End}}
\newcommand{\Idd}{\mathrm{Id}}
\newcommand{\yon}{\mathrm{y}}
\newcommand{\Psk}[1]{\mathrm{Pst}(#1)}
\newcommand{\PskB}{\Psk{\bcatE}}
\newcommand{\Mod}{\mathrm{Mod}}
\newcommand{\Mon}{\mathrm{Mon}}
\newcommand{\Bdj}{\mathrm{Bdj}}
\newcommand{\Palg}{P\text{-}\mathrm{Blg}}
\newcommand{\Qalg}{Q\text{-}\mathrm{Blg}}
\newcommand{\Left}{\mathit{Left}}
\newcommand{\RIGht}{\mathit{Right}}
\newcommand{\conv}{\mathbin{\widehat\oplus}}
\newcommand{\term}{T}
\newcommand{\colim}{\operatorname{colim}}

\newcommand{\overbar}[1]{\mkern 1.5mu\overline{\mkern-1.5mu#1\mkern-1.5mu}\mkern 1.5mu}

 \renewcommand{\vec}[1]{\overbar{#1}}
\newcommand{\seq}[1]{\boldsymbol{#1}}


\newcommand{\obj}[1]{\mathbb{#1}}
\newcommand{\objX}{\obj{X}}
\newcommand{\objXi}{\obj{X}_i}
\newcommand{\objXone}{\obj{X}_1}
\newcommand{\objXtwo}{\obj{X}_2}
\newcommand{\objY}{\obj{Y}}
\newcommand{\objYone}{{\objY}_1}
\newcommand{\objYtwo}{{\objY}_2}
\newcommand{\objU}{\obj{U}}
\newcommand{\objV}{\obj{V}}
\newcommand{\objZ}{\obj{Z}}
\newcommand{\objC}{\obj{C}}
\newcommand{\objK}{\obj{K}}
\newcommand{\objT}{\obj{T}}


\newcommand{\cat}[1]{\mathbb{#1}}
\newcommand{\catB}{\cat{B}}
\newcommand{\catC}{\cat{C}}
\newcommand{\catK}{\cat{K}}
\newcommand{\catM}{\cat{M}}
\newcommand{\catN}{\cat{N}}
\newcommand{\catP}{\cat{P}}
\newcommand{\catQ}{\cat{Q}}
\newcommand{\catR}{\cat{R}}
\newcommand{\catX}{\cat{X}}
\newcommand{\catY}{\cat{Y}}
\newcommand{\catZ}{\cat{Z}}
\newcommand{\catU}{\cat{U}}
\newcommand{\catW}{\cat{W}}


\newcommand{\bcat}[1]{\mathcal{#1}}
\newcommand{\bcatE}{\bcat{E}}
\newcommand{\bcatF}{\bcat{F}}
\newcommand{\bcatG}{\bcat{G}}


\newcommand{\Set}{\mathbf{Set}}
\newcommand{\sSet}{\mathbf{sSet}}

\newcommand{\Bb}{\mathbf{Bb}}
\newcommand{\Bicat}{\mathbf{Bicat}}
\newcommand{\Ext}{\mathrm{Ext}}

\newcommand{\Trib}{\mathbf{Trib}}
\newcommand{\Cart}{\mathbf{Cart}}
\newcommand{\CoCart}{\mathbf{CoCart}}
\newcommand{\flCat}{\mathbf{Cat}^{\mathrm{flim}}}
\newcommand{\Clan}{\mathbf{Clan}}
\newcommand{\CoClan}{\mathbf{CoClan}}
\newcommand{\coClan}{\mathbf{coClan}}
\newcommand{\coTrib}{\mathbf{coTrib}}
\newcommand{\FTribe}{\mathbf{FTribe}}
\newcommand{\HTribe}{\mathbf{HTribe}}
\newcommand{\CBT}{\mathbf{CBT}}
\newcommand{\CatV}{\Cat_{\catV}}
\newcommand{\CBTV}{\CBT_{\catV}}

\newcommand{\bfP}{\mathbf{P}}

\newcommand{\MCatV}{\mathbf{MCat}_\catV}
\newcommand{\MCBTV}{\mathbf{MCBT}_\catV}
\newcommand{\LMCatV}{{}^l\mathbf{MCat}_\catV}
\newcommand{\LMCBTV}{{}^l\mathbf{MCBT}_\catV}
\newcommand{\SMCatV}{\mathbf{SMCat}_\catV}
\newcommand{\SMCBTV}{\mathbf{SMCBT}_\catV}
\newcommand{\LSMCatV}{{}^l\mathbf{SMCat}_\catV}
\newcommand{\LSMCBTV}{{}^l\mathbf{SMCBT}_\catV}

\newcommand{\MatV}{\mathbf{Mat}_\catV}
\newcommand{\LocV}{\mathbf{Loc}_\catV}
\newcommand{\DistV}{\mathbf{Dist}_\catV}
\newcommand{\MDistV}{\mathbf{MDist}_\catV}
\newcommand{\TDistV}{T\mathbf{Dist}_\catV}
\newcommand{\LMDistV}{{}^l\mathbf{MDist}_\catV}
\newcommand{\LSMDistV}{{}^l\mathbf{SMDist}_\catV}
\newcommand{\SMDistV}{\mathbf{SMDist}_\catV}
\newcommand{\VSym}{\mathbf{Seq}_\catV}
\newcommand{\SDistV}{\Sigma_\star\mathbf{Dist}_\catV}
\newcommand{\BFunV}{\mathbf{BFun}_\catV}
\newcommand{\VOpd}{\catV\text{-}\mathbf{Opd}}

\newcommand{\RigV}{\mathbf{Rig}_\catV}
\newcommand{\SRigV}{\mathbf{SRig}_\catV}
\newcommand{\LRigV}{{}^l\mathbf{Rig}_\catV}
\newcommand{\LSRigV}{{}^l\mathbf{SRig}_\catV}
\newcommand{\FSRig}{\mathbf{FSRig}_\catV}
\newcommand{\rig}[1]{{#1}^*}
\newcommand{\rigX}{\rig{\catX}}
\newcommand{\rigY}{\rig{\catY}}
\newcommand{\rigZ}{\rig{\catZ}}


\newcommand{\moncatC}{\mathcal{C}}
\newcommand{\moncatV}{\mathcal{V}}
\newcommand{\catV}{\mathcal{V}}


\newcommand{\EM}{\mathrm{EM}}
\newcommand{\emB}{\EM(\bcatE)}
\newcommand{\Bim}{\mathrm{Bim}}
\newcommand{\bimE}{\Bim(\bcatE)}
\newcommand{\bimF}{\Bim(\bcatF)}
\newcommand{\Mnd}{\mathrm{Mnd}}
\newcommand{\MndB}{\Mnd(\bcatE)}
\newcommand{\MndC}{\Mnd(\bcatF)}


\newcommand{\Id}{\mathrm{Id}}
\newcommand{\IdB}{\Id_\bcatE}
\newcommand{\IdC}{\Id_\catZ}

\newcommand{\myJ}{\mathrm{J}}
\newcommand{\JB}{\myJ_\bcatE}
\newcommand{\JC}{\myJ_\bcatF}





\section[Aspects of Duality : \textit{by André Joyal}]{Aspects of Duality \\ \textit{by André Joyal}}

The results of this appendix are folklore and where possible, we will provide a reference.
\subsection{On certain monoidal transformations}

Let $\mathcal{C}$ be a symmetric monoidal category.
If $C$ is a dualisable object 
in $\mathcal{C}$, with dual object $C^\star $, let us denote 
by $\eta_C:I\to C\otimes C^\star $
and $\epsilon_C:C^\star \otimes C\to I$
the unit and counit of the duality.
If $D$ is another dualisable object,
then the  {\it dual} $f^\star:D^\star\to C^\star$
of a morphism $f:C\to D$ is defined to
be  the composite
$$
 \xymatrix{
D^\star \ar[rr]^(0.4){D^\star\otimes \eta_C} && 
D^\star \otimes C \otimes C^\star \ar[rr]^{D^\star \otimes f \otimes C^\star } && D^\star \otimes D \otimes C^\star  \ar[rr]^(0.6){\epsilon_D\otimes C^\star} && C^\star
}$$
Let us say that a morphism $g:C^\star \to D^\star$
 {\it respects} the morphism $f:C\to D$ if the following 
diagrams commutes

 $$
 \xymatrix{
 \ar[r]^(0.6){\epsilon_C} {C^\star}\otimes C   \ar[d]_{g\otimes f} & I \ar@{=}[d] \\
  \ar[r]^(0.6){\epsilon_D} {D^\star}\otimes D  & I
  } \quad \quad
     \xymatrix{
 I  \ar@{=}[d]  \ar[r]^(0.4){\eta_C}  & C\otimes {C^\star}   \ar[d]^{f\otimes g} \\
 I \ar[r]^(0.4){\eta_D}  &  D\otimes {D^\star}  
  }  $$

\begin{lemma} \label{respect}
If a morphism $g:C^\star \to D^\star$ respects a morphism
$f:C\to D$, then $g\circ f^\star =1_{D^\star}$ 
and $f^\star \circ g =1_{C^\star}$. 
Hence the morphisms $f$ and $g$ are invertible.
 \end{lemma}

\begin{proof} Let us compute $g\circ f^\star$.
The following diagram commutes by naturality: 
$$
 \xymatrix{
D^\star \ar[rr]^(0.4){D^\star\otimes \eta_C} && 
D^\star \otimes C \otimes C^\star \ar[rr]^{D^\star \otimes f \otimes C^\star } 
\ar[d]^{D^\star \otimes C \otimes g} && D^\star \otimes D \otimes C^\star  \ar[rr]^(0.6){\epsilon_D\otimes C^\star}
\ar[d]^{D^\star \otimes D \otimes g} && C^\star \ar[d]^g \\
&&D^\star \otimes C \otimes D^\star \ar[rr]^{D^\star \otimes f \otimes D^\star } &&  D^\star \otimes D \otimes D^\star  
\ar[rr]^(0.6){\epsilon_D\otimes D^\star}  & & D^\star 
}$$ 
 It follows that the morphism $g\circ f^\star$ is the composite
 $$
 \xymatrix{
D^\star \ar[rr]^(0.4){D^\star\otimes \eta_C} && 
D^\star \otimes C \otimes C^\star \ar[rr]^{D^\star \otimes f \otimes g} && D^\star \otimes D \otimes D^\star  \ar[rr]^(0.6){\epsilon_D\otimes D^\star} && D^\star
}$$
But we have $(f\otimes g)\eta_C=\eta_D$, since $g$ respects $f$.
Hence the morphism $g\circ f^\star$ is the composite
$$
 \xymatrix{
D^\star \ar[rr]^(0.4){D^\star\otimes \eta_D} 
 && D^\star \otimes D \otimes D^\star  \ar[rr]^(0.6){\epsilon_D\otimes D^\star} && D^\star
}$$
But $(\epsilon_D\otimes D^\star) (D^\star\otimes \eta_D)=1_{D^\star}$
by the duality between $D$ and $D^\star$.
This shows that $g\circ f^\star=1_{D^\star}$.
The proof that $f^\star \circ g =1_{C^\star}$ is similar.
\end{proof}

\begin{lemma} \label{invertnat}
Let $\alpha:F\to G$ be a monoidal natural
transformation between symmetric monoidal functors
$F, G:\mathcal{C}\to \mathcal{D}$ between  symmetric
monoidal categories $\mathcal{C}=(\mathcal{C},\otimes, I)$
and $\mathcal{D}=(\mathcal{D},\otimes, J)$.
If the monoidal category $\mathcal{C}$
is compact closed, then $\alpha$ is invertible.
\end{lemma}

\begin{proof}
Let us show that the map $\alpha_C:F(C)\to G(C)$
is invertible for every object $C\in \mathcal{C}$.
The object $C$
has a dual $C^\star$, since the category $\mathcal{C}$
is compact closed.
 Let $\eta_C:I\to C\otimes C^\star $
and $\epsilon_C:C^\star \otimes C\to I$
be the unit and counit of the duality.
The object $F(C)$ has then a dual $F(C)^\star :=F(C^\star)$.
The unit $\eta_{F(C)}: J \to F(C)\otimes F(C^\star )$
is defined to be the composite
$$
 \xymatrix{
J \ar[r]^{\simeq} & 
F(I) \ar[rr]^{F(\eta_C)} && F(C \otimes C^\star)   \ar[r]^(0.45){\simeq} & F(C)\otimes F(C^\star )
}$$
and the counit $\epsilon_{F(C)}: F(C^\star ) \otimes F(C)\to J $
is defined to be the composite
$$
 \xymatrix{
F(C^\star )\otimes F(C)   \ar[r]^{\simeq} & 
F(C^\star \otimes C) \ar[rr]^{F(\epsilon_C)} &&    F(I) \ar[r]^(0.5){\simeq} & J
}$$
Similarly,
the object $G(C)$ has a dual $G(C)^\star :=G(C^\star)$.
The unit $\eta_{G(C)}: J \to G(C)\otimes G(C^\star) $
is defined to be the composite
$$
 \xymatrix{
J \ar[r]^{\simeq} & 
G(I) \ar[rr]^{G(\eta_C)} && G(C \otimes C^\star)   \ar[r]^(0.45){\simeq} & G(C)\otimes G(C^\star)
}$$
and the counit $\epsilon_{G(C)}: G(C^\star) \otimes G(C)\to J $
is defined to be the composite
$$
 \xymatrix{
G(C^\star )\otimes G(C)   \ar[r]^{\simeq} & 
G(C^\star \otimes C) \ar[rr]^{G(\epsilon_C)} &&    G(I) \ar[r]^(0.5){\simeq} & J
}$$
Let us show that the morphism $\alpha_{C^\star}:F(C^\star )\to G(C^\star)$
respects the morphism $\alpha_{C}:F(C)\to G(C)$.
The following diagram commutes, since the natural
transformation $\alpha:F\to G$ is monoidal

 $$
  \xymatrix{
J \ar@{=}[d]  \ar[r]^{\simeq} & 
F(I)\ar[d]^{\alpha_I}  \ar[rr]^{F(\eta_C)} && F(C \otimes C^\star) \ar[d]^{\alpha_{C\otimes C^\star}}   \ar[rr]^(0.45){\simeq} && F(C)\otimes F(C^\star ) \ar[d]^{\alpha_{C}\otimes \alpha_{C^\star}}\\
J \ar[r]^{\simeq} & 
G(I) \ar[rr]^{G(\eta_C)} && G(C \otimes C^\star)   \ar[rr]^(0.45){\simeq} && G(C)\otimes G(C^\star )
}$$
 
Hence the following square commutes

 $$ \xymatrix{
 J \ar@{=}[d]  \ar[rr]^(0.4){\eta_{F(C)}}  && F(C)\otimes {F(C^\star)}   \ar[d]^{\alpha_C\otimes \alpha_{C^\star}} \\
 J \ar[rr]^(0.4){\eta_{G(D)}}  &&  G(C)\otimes {G(D^\star)}  
  }  $$

Similarly, the following square commutes
$$\xymatrix{
 \ar[rr]^(0.6){\epsilon_{F(C)}} {F(C^\star)}\otimes F(C)   \ar[d]_{\alpha_{C^\star} \otimes \alpha_C} && J \ar@{=}[d] \\
  \ar[rr]^(0.6){\epsilon_{G(D)}} G({D^\star})\otimes G(D)  && J
  }$$
This shows that the morphism $\alpha_{C^\star}:F(C^\star )\to G(C^\star)$
respects the morphism $\alpha_{C}:F(C)\to G(C)$.
It then follows by Lemma \ref{respect} that $\alpha_C$
is invertible.
\end{proof}

If $\mathcal{C}$ and $\mathcal{D}$
are symmetric monoidal categories,
let us denote by $Hom_{SM}(\mathcal{C},\mathcal{D})$
the category of symmetric monoidal functors
$\mathcal{C}\to \mathcal{D}$.
The category $Hom_{SM}(\mathcal{C},\mathcal{D})$
is symmetric monoidal.

A different proof of the following proposition appears in \cite[Prop. 1.13]{DM}

\begin{proposition} \label{homgrp}
If $\mathcal{C}$ is a compact closed
symmetric monoidal category, then 
the symmetric monoidal category $Hom_{SM}(\mathcal{C},\mathcal{D})$
is a groupoid for every symmetric monoidal category $\mathcal{D}$. 
\end{proposition}

\begin{proof} Let $\alpha:F\to G$
be a morphism in the category $Hom_{SM}(\mathcal{C},\mathcal{D})$.
The natural transformation $\alpha$ is invertible by Lemma
\ref{invertnat}. Its inverse $\alpha^{-1}:G\to F$
is monoidal (by a general result).
Thus, $Hom_{SM}(\mathcal{C},\mathcal{D})$
is a groupoid.
 \end{proof}

 \subsection{On the compact closed
 symmetric monoidal category free on one generator } 
 Let me denote by $\mathcal{B}$ the compact closed symmetric
 monoidal category freely generated by one object $U\in 
 \mathcal{B}$. By definition, for every 
 compact closed symmetric monoidal category 
 $\mathcal{C}$ and every object $C\in  \mathcal{C}$
 there exists a symmetric monoidal functor
 $F:\mathcal{B}\to \mathcal{C}$ such that $F(U)=C$,
and the functor $F$ is unique
 up to unique isomorphism: if $G:B\to \mathcal{C}$
 is another functor such that $G(U)=C$,
 then there exists a unique monoidal natural isomorphism
 $\alpha: F\to G$ such that $\alpha_U=1_C$.

\medskip

If $\mathcal{C}$ is a category, 
then the subcategory of invertible morphisms
of $\mathcal{C}$ is a groupoid
called the {\it core} of $\mathcal{C}$.
I will denote the core of $\mathcal{C}$
by $\mathcal{C}^{cor}$.
The core of a symmetric monoidal
category $\mathcal{C}$ is a symmetric monoidal
subcategory of $\mathcal{C}$.

\medskip

I will use the following construction
in the proof of the next propsition.
 Let me denote by $J$ the groupoid 
freely generated by one isomorphism $i:0\to 1$.
If $\mathcal{C}$ is a category
then an object of the category $\mathcal{C}^J$
is an isomorphism $f$ in $\mathcal{C}$.
The source and target
functors  $s,t:\mathcal{C}^J\to \mathcal{C}$
are connected by a natural isomorphism
$h:s\to t$ defined by putting $h(f)=f:s(f)\to t(f)$.
The category $\mathcal{C}^J$ is symmetric
monoidal if $\mathcal{C}$ is symmetric
monoidal. Moreover, the source and target
functors  $s,t:\mathcal{C}^J\to \mathcal{C}$
and the
natural transformation $h:s\to t$
are symmetric monoidal.
The category $\mathcal{C}^J$
is compact closed if $\mathcal{C}$
is compact closed, since the
functor $s:\mathcal{C}^J\to \mathcal{C}$
is an equivalence of symmetric monoidal
categories.

\begin{proposition}
	\label{1-D-CH}
 Let $\mathcal{B}$ the compact closed symmetric
 monoidal category freely generated by one object $U\in 
 \mathcal{B}$.
If $\mathcal{C}$ is a compact closed
symmetric monoidal category, 
then the evaluation functor
$$e_U: Hom_{SM}(\mathcal{B},\mathcal{C})\to 
\mathcal{C}$$
defined by putting $ev_U(F):=F(U)$
takes its values in the core of  $\mathcal{C}$.
Moreover, the induced functor 
$$e'_U: Hom_{SM}(\mathcal{B},\mathcal{C})\to 
\mathcal{C}^{cor}$$
is an equivalence of symmetric monoidal categories.
\end{proposition}

\begin{proof} The category $Hom_{SM}(\mathcal{B},\mathcal{C})$ is a groupoid by Proposition
 \ref{homgrp}. Hence the functor $e_U$
 takes its values in the core of $\mathcal{C}$.
 Let us show that the induced functor $e'_U$
 is an equivalence of categories.
 For every object $C\in \mathcal{C}$
 there exists a symmetric monoidal
 functor $F:\mathcal{B}\to \mathcal{C}$ such that $F(U)=C$,
 since $\mathcal{C}$ is compact closed
 and $\mathcal{B}$ is compact closed and freely
 generated by the object $U\in \mathcal{B}$.
 We then have $e'_U(F):=e_U(F):=F(U)=C$.
 We have proved that the functor $e'_U$ is surjective
 on objects. Let us show that the functor
 $e'_U$ is fully faithful.  If $F,G:\mathcal{B}\to \mathcal{C}$
 are symmetric monoidal functors, let us show
 that for every isomorphism 
 $f:F(U)\to G(U)$ there exists a unique monoidal natural
isomorphism $\alpha:F\to G$ such that $\alpha_U=f$.
We shall first prove the existence of $\alpha$.
The symmetric
monoidal category $\mathcal{C}^J$
is compact closed, since the 
symmetric monoidal category
$\mathcal{C}$ is compact closed by hypothesis.
The isomorphism $f$ is an object in $\mathcal{C}^J$.
By the freeness of $\mathcal{B}$,
 there exists a symmetric monoidal
functor
$$H:\mathcal{B}\to \mathcal{C}^J$$
such that $H(U)=f$.
We have $s H(U)=s(f)=F(U)$, since $f:F(U)\to G(U)$.
The functor $sH:\mathcal{B}\to \mathcal{C}$
is symmetric monoidal, since the
functors $H$ and $s$ are.
Thus, there exists a unique monoidal natural isomorphism 
$\rho:F\to sH$ such that $\rho_U=1_{F(U)}$.
Similarly, if $t:\mathcal{C}^J\to \mathcal{C}$ is 
the target functor, then $t H(U)=t(f)=G(U)$.
Thus, there exists a unique monoidal natural isomorphism 
$\lambda:t H\to G$ such that $\lambda_U=1_{G(U)}$.
If $h:s\to t$ is the canonical isomorphism, then
the composite $\alpha:=\lambda h \rho$
is a monoidal natural isomorphism $\alpha:F\to G$
$$
\xymatrix{
F\ar[r]^\rho & sH \ar[r]^{h\circ H}  & t H \ar[r]^{\lambda} & G
}$$
We have $\alpha_U=f$, since $\rho_U=1_{F(U)}$,
$(h\circ H)_U=h({H(U)})=h(f)=f$ and $\lambda_U=1_{G(U)}$.
The existence of $\alpha:F\to G$ is proved.
Let us show that $\alpha$ is unique.
Let $\beta:F\to G$ a monoidal
natural isomorphism such that $\beta_U=f$.
Then $\gamma:=\beta^{-1}\alpha:F\to F$
is a monoidal natural isomorphism such that $\gamma_U=1_U$. It follows that $\gamma =1_F$,
since $\mathcal{B}$ is freely generated by the object $U\in 
 \mathcal{B}$.
 We have proved that the functor 
$e'_U: Hom_{SM}(\mathcal{B},\mathcal{C})\to 
\mathcal{C}^{cor}$
 is fully faithful. It is thus an equivalence
 of categories, since it is surjective on objects.
 It is also an equivalence of symmetric
 monoidal categories, since it is a 
 symmetric monoidal functor.
 \end{proof}


%% file: LocMdlCats.tex
\section[Localization in model categories]{Localization in model categories}
\label{loc-Mdl-Cats}
In this appendix we recall the notion of a \emph{left Bousfield localization} of a model category and also recall an existence result of the same.

\begin{df}
	Let $\M$ be a model category and let $\S$ be a class of maps in $\M$.
	The left Bousfield localization of $\M$ with respect to $\S$
	is a model category structure $L_\S\M$ on the underlying category of $\M$
	such that
	\begin{enumerate}
		\item The class of cofibrations of $L_\S\M$ is the same as the
		class of cofibrations of $\M$.
		
		\item A map $f:A \to B$ is a weak equivalence in $L_\S\M$ if it is an $\S$-local equivalence,
		namely, for every fibrant $\S$-local object $X$, the induced map on homotopy
		function complexes
		\[
		f^\ast:Map_{\M}^h(B, X) \to Map_{\M}^h(A, X)
		\]
		is a weak homotopy equivalence of simplicial sets. Recall
		that an object $X$ is called fibrant $\S$-local if $X$ is fibrant
		in $\M$ and for every element
		$g:K \to L$ of the set $\S$, the induced map on
		homotopy function complexes
		\[
		g^\ast:Map_{\M}^h(L, X) \to Map_{\M}^h(K, X)
		\]
		is a weak homotopy equivalence of simplicial sets.
		
	\end{enumerate}
	
\end{df}

We recall the following theorem
which will be the main tool in the construction of the
desired model category. This theorem first appeared in an unpublished work \cite{smith}
but a proof was later provided by Barwick in \cite{CB1}.
\begin{thm} \cite[Theorem 2.11]{CB1}
	\label{local-tool}
	If $\M$ is a left-proper, combinatorial model category and $\S$ is a small
	set of homotopy classes of morphisms of $\M$, the left Bousfield localization $L_\S\M$ of
	$\M$ along any set representing $\S$ exists and satisfies the following conditions.
	\begin{enumerate}
		\item The model category $L_\S\M$ is left proper and combinatorial.
		\item As a category, $L_\S\M$ is simply $\M$.
		\item The cofibrations of $L_\S\M$ are exactly those of $\M$.
		\item The fibrant objects of $L_\S\M$ are the fibrant $\S$-local objects $Z$ of $\M$.
		\item The weak equivalences of $L_\S\M$ are the $\S$-local equivalences.
	\end{enumerate}
\end{thm}